\documentclass{amsart}

\makeatletter
 \def\LaTeX{\leavevmode L\raise.42ex
   \hbox{\kern-.3em\size{\sf@size}{0pt}\selectfont A}\kern-.15em\TeX}
\makeatother

\newcommand{\BibTeX}{{\rm B\kern-.05em{\sc
i\kern-.025emb}\kern-.08em\TeX}}

\newtheorem{thm}{Theorem}[section]
\newtheorem{lem}[thm]{Lemma}

\newtheorem{defn}[thm]{Definition}

\numberwithin{equation}{section}

\begin{document}

\title[New estimates of $n$-widths on Compact Manifolds]
{Estimates of Kolmogorov, Gelfand and linear $n$- widths on Compact Riemannian Manifolds}

\maketitle
\begin{center}

\author{Isaac Z. Pesenson }\footnote{ Department of Mathematics, Temple University,
 Philadelphia,
PA 19122; pesenson@temple.edu }

\end{center}

\begin{abstract}
We determine lower and exact  estimates of  Kolmogorov, Gelfand and linear  $n$-widths of unit balls in Sobolev  norms in $L_{p}$-spaces on 
compact Riemannian 
manifolds.   As it was shown  by us previously these lower estimates are exact asymptotically in the case of compact homogeneous manifolds.  The proofs   rely  on two-sides  estimates for the near-diagonal localization  of   kernels of functions of elliptic operators.
\end{abstract}



% \keywords{Compact homogeneous manifold, wavelets, Laplace operator, eigenfunctions,}
% \subjclass[2000]{ 43A85; 42C40; 41A17;
%Secondary 41A10}

\section{Introduction and the main results}

\bigskip

The goal of the paper is to determine lower and exact  estimates of  Kolmogorov, Gelfand and linear $n$-widths of unit balls in
Sobolev  norms in $L_{p}({\bf M})$-spaces on a compact  connected Riemannian 
manifold ${\bf M}$.   

Let us recall \cite{LGM}, \cite{pin} that for a given subset $H$ of a normed linear space $Y$, the Kolmogorov $n$-width
$d_{n}(H,Y)$ is defined as
$$
d_{n}(H,Y)=\inf_{Z_{n}}\sup_{x\in H}\inf_{z\in Z_{n}}\|x-z\|_{Y}
$$
where $Z_{n}$ runs over all $n$-dimensional subspaces of $Y$. The linear $n$-width $\delta_{n}(H,Y)$  is defined as 
$$
\delta_{n}(H,Y)=\inf _{A_{n}}\sup_{x\in H}\|x-A_{n}x\|_{Y}
$$
where $A_{n}$ runs over all bounded operators $A_{n}: Y\rightarrow Y$ whose range has dimension $n$.
The Gelfand  $n$-width of a subset $H$ in a linear space $Y$ is defined by
$$
d^{n}(H,Y)=\inf_{Z^{n}}\sup\{\|x\|: \>\>x\in H\cap Z^{n}\},
$$
where the infimum is taken over all subspaces $Z^{n}\subset Y$ of codimension $\leq n$.
The width $d_{n}$ characterizes the best approximative possibilities by approximations by $n$-dimensional subspaces, the width $\delta_{n}$ characterizes the best approximative possibilities  of any $n$-dimensional linear method. The width $d^{n}$ plays a key role in questions about interpolation and reconstruction of functions.

In our paper the notation $S_{n}$ will stay for either Kolmogorov $n$-width $d_{n}$ or linear $n$-width $\delta_{n}$;
 the notation $s_{n}$ will be used for either $d_{n}$ or Gelfand $n$-width $d^{n};
\>\> S^{n}$ will be used for either $d_{n}, \>\> d^{n},$ or $\delta_{n}$.

If $\gamma \in \bf R$, we write $S^n(H, Y)  \ll n^{\gamma},\>\>n\in \mathbb{N},$ to mean that one has the upper estimate $S^n(H, Y) \geq Cn^{\gamma}$ 
 where  $C$ is independent  of $n$.  We say that  one has 
 the lower estimate if $S^n(H, Y)\gg cn^{\gamma}$ for $n > 0$  where $c>0$ is independent of $n$. In the case we have both estimates  we write $S^n(H, Y) \asymp cn^{\gamma}$ and call it the exact estimate.
More general, for two functions $f(t),\>g(t)$ notation $f(t)\asymp g(t)$ means existence of two unessential positive constants $c, C$ for which $cg(t)\leq f(t)\leq Cg(t)$ for all admissible $t$.

Let $({\bf M},g)$ be a smooth, connected, compact Riemannian manifold without boundary
with  Riemannian measure $dx$.  
Let  $L_{q}({\bf M})=L_{q}({\bf M}),\> 1\leq q\leq\infty,$ be the regular Lebesgue  space constructed with the Riemannian density.
Let $L$ be an elliptic smooth second-order 
differential operator $L$ which is self-adjoint and positive definite in $L_{2}({\bf M})$.   For such an operator all the powers $L^{r}, \>\>r>0,$ 
are well defined on $C^{\infty}({\bf M})\subset L_{2}({\bf M})$ and continuously map $C^{\infty}({\bf M})$ into itself. 
Using duality every operator  $L^{r}, \>\>r>0,$ can be extended  to distributions on ${\bf M}$.
The Sobolev space $W_{p}^{r}=W_{p}^{r}({\bf M}),\> 1\leq p\leq\infty, \>\> r>0,$ is defined as the space of all 
$f\in L_{p}({\bf M}), 1\leq p\leq \infty$, for which the following graph norm is finite
\begin{equation}
\label{Sob}
\|f\|_{W^{r}_{p}({\bf M})}=\|f\|_{p}+\|L^{r/2} f\|_{p}.
\end{equation}

Our objective is to obtain asymptotic estimates of $S_{n}\left(H, L_{q}({\bf M})\right)$,  where  $H$   is the unit ball $B^r_p({\bf M})$ in the Sobolev space 
$W_{p}^{r}=W_{p}^{r}({\bf M}), 1\leq p\leq\infty,\>\>r>0,$   Thus,
$$
B^r_p=B^r_p({\bf M})=\left\{f\in W_{p}^{r}({\bf M})\>: \> \|f\|_{W_{p}^{r}({\bf M})}\leq 1\right\}.
$$
It is  important to remember that in all our considerations the inequality
 \begin{equation}\label{basicineq}
\frac{r}{s}>\left(\frac{1}{p}-\frac{1}{q}\right)_+
\end{equation}
with $s=dim \>{\bf M}$ will be assumed. Thus, by the Sobolev embedding theorem the set  $B^{r}_{p}({\bf M})$ is a  subset of $L_{q}({\bf M})$. Moreover, since ${\bf M}$ is compact,  the Rellich-Kondrashov theorem implies that the embedding of $B^{r}_{p}({\bf M})$ into $L_{q}({\bf M})$  will be compact.

  Our main result is the following  theorem which is proved in     section \ref{Proof}.

\begin{thm}\label{Main}
For any compact Riemannian manifold ${\bf M}$ of dimension $s$, any elliptic second-order smooth operator $L$,  any $1\leq p,\ q\leq\infty,$   if  $S_{n}$ is either of $d_{n}$ or $\delta_{n}$ and $S^{n}$ is either of $d_{n}$, $\delta_{n}$ or $d^{n}$ then the following holds for any $r$ that satisfies (\ref{basicineq}):

\begin{enumerate}

\item  if $1\leq q\leq p\leq \infty$ then

\begin{equation}
S^n\left(B^r_p({\bf M}),L_q({\bf M})\right) \gg n^{-\frac{r}{s}},
\end{equation}

\item if $1\leq p\leq q\leq 2$ then

\begin{equation}
S_{n}(B^r_p({\bf M}),L_q({\bf M}))  \gg n^{-\frac{r}{s}+\frac{1}{p}-\frac{1}{q}},
\end{equation}

and

\begin{equation}
d^{n}\left(B^r_p({\bf M}),L_q({\bf M})\right)  \gg n^{-\frac{r}{s}},
\end{equation}

\item if $2\leq p\leq q\leq \infty$ then

\begin{equation}
d_{n}\left(B^r_p({\bf M}),L_q({\bf M})\right)  \gg n^{-\frac{r}{s}}, 
\end{equation}

\begin{equation}
d^{n}(B^r_p({\bf M}),L_q({\bf M}))  \gg n^{-\frac{r}{s}+\frac{1}{p}-\frac{1}{q}},
\end{equation}

\begin{equation}
\delta_{n}(B^r_p({\bf M}),L_q({\bf M}))  \gg n^{-\frac{r}{s}+\frac{1}{p}-\frac{1}{q}},
\end{equation}

\item if $1\leq p\leq 2\leq q\leq \infty$ then

\begin{equation}
d_{n}(B^r_p({\bf M}),L_q({\bf M}))  \gg n^{-\frac{r}{s}+\frac{1}{p}-\frac{1}{2}},
\end{equation}

\begin{equation}
d^{n}(B^r_p({\bf M}),L_q({\bf M}))  \gg n^{-\frac{r}{s}+\frac{1}{2}-\frac{1}{q}},
\end{equation}

\begin{equation}
\delta_{n}(B^r_p({\bf M}),L_q({\bf M}))  \gg max \left(   n^{-\frac{r}{s}+\frac{1}{2}-\frac{1}{q}}, n^{-\frac{r}{s}+\frac{1}{q}-\frac{1}{2}}\right ).
\end{equation}

\end{enumerate}

\end{thm}

As it was shown in \cite{gp3} all the estimates of this theorem are exact if ${\bf M}$ is a compact homogeneous manifold. Let's remind that every compact homogeneous manifold is of the form $G/H$ where $G$ is a compact Lie group and $H$ is its closed subgroup. 
For  compact  homogeneous manifolds we obtained in \cite{gp3}  exact asymptotic estimates for $d_{n}$ and $\delta_{n}$ for all $1\leq p,q\leq \infty$ and some restrictions on $r$.

To compare our lower estimates with known upper estimates, let us remind an inequality which was proved in our previous papers \cite{gp2}, \cite{gp3}. 
\begin{thm}
\label{basic}

For any compact Riemannian manifold, any $L$, any $1\leq p,q\leq\infty,$ and any  $\> r$  which satisfies (\ref{basicineq})
 if  $S_{n}$ is either of $d_{n}$ or $\delta_{n}$ then the following holds 
\begin{equation}
\label{basicway}
S_n(B^r_p({\bf M}),L_q({\bf M})) \ll n^{-\frac{r}{s}+\left(\frac{1}{p}-\frac{1}{q}\right)_+}.
\end{equation}

\end{thm}

By comparing these two theorems we obtain the following exact estimates.

\begin{thm}
For any compact Riemannian manifold ${\bf M}$ of dimension $s$, any elliptic second-order smooth operator $L$,  any $1\leq p,\ q\leq\infty,$   if  $S_{n}$ is either of $d_{n}$ or $\delta_{n}$  then the following holds for any $r$ which satisfies (\ref{basicineq})
\begin{enumerate}

\item if  $1\leq q\leq p\leq \infty$, then

$$
S_{n}\left(B_{p}^{r}({\bf M}), L_{q}({\bf M})\right)\asymp n^{-\frac{r}{s}},
$$

\item if $1\leq p\leq q\leq 2$, then

$$
S_{n}\left(B_{p}^{r}({\bf M}), L_{q}({\bf M})\right)\asymp n^{-\frac{r}{s}+\frac{1}{p}-\frac{1}{q}},
$$

\item if $2\leq p\leq q\leq \infty$, then 

$$
\delta_{n}\left(B_{p}^{r}({\bf M}), L_{q}({\bf M})\right)\asymp n^{-\frac{r}{s}+\frac{1}{p}-\frac{1}{q}}.
$$

\end{enumerate}

\end{thm}

Our results could be carried over to Besov spaces on manifolds using general results about interpolation of compact operators \cite{T}. 

Our results generalize some of the known estimates for
the particular case in which $\bf M$ is a compact symmetric space of rank one which were obtained in 
papers \cite{BKLT} and \cite{brdai}.
They, in turn generalized and extended results from \cite{BirSol}, \cite{Ho}, \cite{Kas}, \cite{Ma}, \cite{Ka1}, \cite{Ka2}, \cite{Kas}.

 {\bf Acknowledgment} 

In July of 2015 at  a BIRS-CMO meeting in Oaxaca Prof. B. Kashin asked me  if  our results with D.Geller \cite{gp3} about $n$-widths on homogeneous compact manifolds can be extended to general compact Riemannian manifolds. This paper gives a partial answer to his question. I am grateful to Prof. B. Kashin for stimulating my interest to this problem.

\section{Kernels on compact Riemannian manifolds}

We consider $({\bf M},g)$ be a smooth, connected, compact Riemannian manifold without boundary
with  Riemannian measure $dx$. 
Let $L$ be the Laplace-Beltrami operator of the metric $g$ which is well defined on $C^{\infty}({\bf M})$. 
We will use the same notation
$L$ for the closure of $L$ from $C^{\infty}({\bf M})$ in $L_{2}({\bf M})$.
This closure is a
self-adjoint non-negative operator on the space $L_{2}({\bf M})$.
The spectrum of this operator, say
$0=\lambda_{0}<\lambda_{1}\leq \lambda_{2}\leq ...$,
is discrete and approaches infinity.  Let
$u_{0}, u_{1}, u_{2}, ...$ be a corresponding
complete system of real-valued orthonormal eigenfunctions, and let
$\textbf{E}_{t}(L),\  t>0,$ be the span of all
eigenfunctions of $L$, whose corresponding eigenvalues
are not greater than $t$.    Since the
operator $L$ is of order two, the dimension
$\mathcal{N}_{t}$ of the space ${\mathbf E}_{t}(L)$ is
given asymptotically by Weyl's formula, \cite{Hor}, 
which says, in sharp form that for some $c > 0$,
\begin{equation}
\label{Weyl}
\mathcal{N}_{t}(L) = \frac{vol({\bf M})\sigma_{s}}{(2\pi)^{s}}t^{s/2} + O(t^{(s-1)/2}),\>\>\>\sigma_{s}=\frac{2\pi^{s/2}}{s\Gamma(s/2)},\>\>\>s=dim\ {\bf M}.
\vspace{.3cm}
\end{equation}
where $s=dim \ {\bf M}$.
Because  $\mathcal{N}_{\lambda_l} = l+1$, we conclude that, for some constants $c_1, c_2 > 0$,
\begin{equation}
\label{lamest}
c_1 l^{2/s} \leq \lambda_l \leq c_2 l^{2/s}  
\end{equation}
for all $l$.
Since $L^m u_l = \lambda_l^m u_l$, and $L^m$ is an elliptic differential
operator of degree $2m$, Sobolev's lemma, combined with the last fact, implies that
for any integer $k \geq 0$, there exist $C_k, \nu_k > 0$ such that
\begin{equation}
\label{ulest}
\|u_l\|_{C^k({\bf M})} \leq C_k (l+1)^{\nu_k}. 
\end{equation}
For a $t>0$ let's consider the function
\begin{equation}
K_{t}(x,y)=\sum_{\lambda_{l}\leq t}u_{l}(x)u_{l}(y)
\end{equation}
which is known as the spectral function associated to $L$. In \cite{Hor} one can find the following estimate
\begin{equation}\label{Hor}
K_{t}(x,x)=\frac{\sigma_{s}}{(2\pi)^{s}}t^{s/2}+O(t^{s-1}) 
\end{equation}
Since 
$$
K_{t}(x,x)=\sum_{0<\lambda_{l}\leq t}\left(u_{l}(x)\right)^{2}=\|K_{t}(x, \cdot)\|_{2}^{2}
$$
estimates (\ref{Weyl}) and (\ref{Hor})   imply that  there exists $0<C_{1}<C_{2}$ such that
\begin{equation}\label{normdiag1}
C_{1}t^{s/2}\leq\sum_{0<\lambda_{l}\leq t}\left(u_{l}(x)\right)^{2}\leq C_{2}t^{s/2}.
\end{equation}
we also note that 
$$
 dim\ \textbf{E}_{t}(L)\asymp t^{s/2},\>\>\>s=dim\ {\bf M}.
$$

\begin{defn}
For each positive integer $J$,  we let 
\begin{equation}\label{defn}
{\mathcal S}_J({\bf R}^+) = \left\{F \in C^J\left([0,\infty)\right): \|F\|_{{\mathcal S}_J}
:= \sum_{i+j \leq J}\left\|\lambda^i \frac{\partial^j}{\partial\lambda^{j} }F\right\|_{\infty} < \infty\right\}.
\end{equation}

\end{defn}
  For a fixed  $t > 0$ if $J$ is sufficiently large, one can use 
(\ref{Weyl}), (\ref{lamest}) and (\ref{ulest}), to show  that the right side of
\begin{equation}
\label{expoutm}
K_t^F(x,y) := \sum_l F(t^2\lambda_l)u_l(x)u_l(y)
\end{equation}
converges uniformly to a continuous function on ${\bf M} \times {\bf M}$, and in fact that for
some $C_t > 0$,
\begin{equation}
\label{ktt}
\|K_t^F\|_{\infty} \leq C_t\|F\|_{{\mathcal S}_{J}}.
\end{equation}
 Using the spectral theorem,
one can define the bounded operator $F(t^{2}L)$ on $L_2({\bf M})$.  In fact, for $f \in L_2({\bf M})$,
\begin{equation}
\label{ft2F}
[F(t^{2}L)f](x) = \int K_t^{F}(x,y) f(y) dy.
\end{equation}
We call $K_t^{F}$ the kernel of $F(t^2L)$.  $F(t^2L)$ maps $C^{\infty}({\bf M})$
to itself continuously, and may thus be extended to be a map on distributions.  In particular
we may apply $F(t^2L)$ to any $f \in L_p({\bf M}) \subseteq L_1({\bf M})$ (where $1 \leq p \leq \infty$), and by Fubini's theorem
$F(t^2L)f$ is still given by (\ref{ft2F}).

For $x,y \in {\bf M}$, let $d(x,y)$ denote the geodesic distance from $x$ to $y$.
We will frequently need the following fact.

\begin{lem}
If $\mathcal{N}> s$, $x \in {\bf M}$ and $t > 0$, then
\begin{equation}
\label{intest}
\int_{\bf M} \frac{1}{\left[1 + (d(x,y)/t)\right]^\mathcal{N}} dy \leq Ct^{s},\>\>\>s = \dim {\bf M},
\end{equation}
with $C$ independent of $x$ or $t$. 
\end{lem}

 \begin{proof}
 Note, that there exist $c_{1}, c_{2}>0$ such that for all $x\in {\bf M}$ and all sufficiently small $r\leq \delta$ one has
 $$
 c_{1}r^{n}\leq |B(x,r)|\leq c_{2}r^{n},
 $$
 and if $r>\delta$
 $$
 c_{3}\delta^{n}\leq |B(x,r)|\leq |\mathbf{M}|\leq c_{4}r^{n}.
 $$

 Fix $x,t$ and let $A_{j}=B(x, 2^{j}t)\setminus B(x, 2^{j-1}t)$, so that, $|A_{j}|\leq c_{4}2^{nj}t^{n}$. Now break the integral into integrals over $B(x,t), A_{1},...$ and noties that $\sum_{j=0}^{\infty}2^{(n-N)j}<\infty$.

 \end{proof}

The following statements can be found in  \cite{gm}-\cite{gp3}.

\begin{lem}
\label{kersize}
Assume  $F \in \mathcal{ S}_{J}(\bf{R}^{+})$ for a sufficiently large $J\in \mathbb{N}$.
For $t > 0$, let $K_t^{F}(x,y)$ be the kernel of $F(t^{2}L)$.  Suppose that  $0 < t \leq 1$.
Then for some $C > 0$,
\begin{equation}
\label{kersizeway}
|K_t^{F}(x,y)| \leq \frac{Ct^{-s}}{\left[ 1+\frac{d(x,y)}{t} \right]^{s+1}},\>\>\>s=dim\>{\bf M},
\end{equation}
for all $t$ and all $x,y \in {\bf M}$.
\end{lem}

\begin{lem}
\label{Lalphok}
Assume  $F \in \mathcal{ S}_{J}(\bf{R}^{+})$ for a sufficiently large $J\in \mathbb{N}$. Consider  $1 \leq \alpha \leq \infty$, with conjugate index $\alpha'$.
There  exists  a constant $C > 0$ such that for all $0<t\leq 1$
\begin{equation}
\label{kint3a}
\left(\int_{{\bf M}}  |K_t^{F}(x,y)|^{\alpha}dy\right)^{1/\alpha} \leq Ct^{-s/\alpha'} \ \ \ \ \ \ \ \ \ \ \ \ \ \ \ \ \ \ \ \ \mbox{for all } x, 
\end{equation}
and
\begin{equation}
\label{kint4}
\left(\int_{{\bf M}}  |K_t^{F}(x,y)|^{\alpha}dx\right)^{1/\alpha} \leq Ct^{-s/\alpha'} \ \ \ \ \ \ \ \ \ \ \ \ \ \ \ \ \ \ \ \ \mbox{for all } y.
\end{equation}
\end{lem}
\begin{proof}  We need only prove (\ref{kint3a}), since $K_t^{F}(y,x) = K_t^{F}(x,y)$.

If $\alpha < \infty$,
(\ref{kint3a}) follows from Lemma \ref{kersize}, which tells us that
\[
\int_{{\bf M}}  |K_t^{F}(x,y)|^{\alpha}dy 
\leq  C \int_{\bf M} \frac{t^{-s\alpha}}{\left[1 + (d(x,y)/t)\right]^{\alpha(s+1)}} dy \leq Ct^{s(1-\alpha)} \]
with $C$ independent of $x$ or $t$, by (\ref{intest}).

If $\alpha = \infty$, the left side of (\ref{kint3a}) is as usual to be interpreted as the $L_{\infty}$ norm
of $h_{t,x}(y) = K_t^{F}(x,y)$.  But in this case the conclusion is immediate from 
Lemma \ref{kersize}.  

This completes the proof.
\end{proof}

\begin{lem} 
If $C_{1}, C_{2} $ are the same as in (\ref{normdiag1}) and  if 
\begin{equation}\label{condition1}
b/a>\left (C_{2}/C_{1}\right)^{2/s},\>\>\>s=dim\ {\bf M},
\end{equation}
then 
\begin{equation}
\sum_{a/t^{2}<\lambda_{l}\leq b/t^{2}}|u_{l}(x)|^{2}\geq \left(C_{1}b^{s/2}-C_{2}a^{s/2}\right)t^{-s}>0,\>\>\>s=dim\ {\bf M}.
\end{equation}

\end{lem}
\begin{proof}
By the inequalities (\ref{normdiag1}) we have

$$
\sum_{a/t^{2}<\lambda_{l}\leq b/t^{2}}|u_{l}(x)|^{2}=\sum_{0<\lambda_{l}\leq b/t^{2}}|u_{l}(x)|^{2}-\sum_{0<\lambda_{l}\leq a/t^{2}}|u_{l}(x)|^{2}\geq 
\left(C_{1}b^{s/2}-C_{2}a^{s/2}\right)t^{-s}.
$$
Lemma is proven.
\end{proof}

\begin{lem} 
For any $0<a<b$ and sufficiently large  $J\in \mathbb{N}$ 
there exists an even function $F$  in $\mathcal{S}_{J}(\bf{R})$ such that $\widehat{F}$ is supported in an $(-\Lambda, \Lambda)$ for some $\Lambda>0$  and the inequality  
$$\label{intineq}
 0<c_{1}\leq |F(\lambda)|\leq c_{2}
$$
holds for all $a\leq \lambda\leq b$ for some $c_{1}, \ c_{2}>0$.

\end{lem}
\begin{proof}
For a sufficiently large $J$  consider an even  function $G \in \mathcal{S}_{J}(\bf{R})$ which is identical one on $(a,\ b)$.  Since Fourier transform maps continuously  $\mathcal{S}_{J}(\bf{R})$ into itself  one can find an even smooth function $\widehat{F}\in \mathcal{S}_{J}(\bf{R})$ which is supported in an $(-\Lambda, \ \Lambda)$ for some $\Lambda>0$ and which is sufficiently close  to the Fourier transform $\widehat{G}\in \mathcal{S}_{J}(\bf{R})$ in the topology of $\mathcal{S}_{J}(\bf{R})$. Clearly, the function $F$ will have all the desired properties. It proves  Lemma.

\end{proof}

\begin{thm}\label{lowerestimate}
For a sufficiently large $J\in \mathbb{N}$ there exists an even  function $F$ in the  space $\mathcal{S}_{J}(\bf{R})$ such that $\widehat{F}$ has support in an $(-\Lambda,  \Lambda)$ and for which 
\begin{equation}
\left( \int_{{\bf M}}  \left|K_{t}^{F}(x,y)\right|^{\alpha}dy\right)^{1/\alpha}\asymp t^{-s/\alpha^{\prime}},\>\>\>\>1/\alpha+1/\alpha^{\prime}=1,\>\>\>0< t\leq 1,
\end{equation}
for all $1\leq \alpha\leq\infty$.
\end{thm}

\begin{proof}
Due to Lemma \ref{Lalphok} we have to prove only the lower estimate. Assume that $0<a<b$ and  satisfy (\ref{condition1}). 
Let $F$ be a function whose existence is proved in the previous Lemma for  this  $(a, b)$. 
For $\alpha=2$ one has 
$$
\int_{{\bf M}} |K^{F}_{t}(x,y)|^{2}dy=\sum_{l}|F(t^{2}\lambda_{l})|^{2}|u_{l}(x)|^{2}\geq \sum _{l: a/t^{2}\leq \lambda_{l}\leq b/t^{2}}|F(t^{2}\lambda_{l})|^{2}|u_{l}(x)|^{2}\geq 
$$
$$
c_{1}^{2}\sum _{l: a/t^{2}\leq \lambda_{l}\leq b/t^{2}}|u_{l}(x)|^{2}\geq c_{1}^{2}\left(C_{1}b^{s/2}-C_{2}a^{s/2}\right)t^{-s}>0.
$$
Using the inequality
\begin{equation}
\|K^{F}_{t}(x, \cdot)\|_{2}^{2}\leq \|K^{F}_{t}(x, \cdot)
\|_{1}\|K^{F}_{t}(x, \cdot)
\|_{\infty},
\end{equation}
and Lemma \ref{Lalphok} for $\alpha=1$ we obtain for $\alpha=1$
$$
\|K^{F}_{t}(x, \cdot)
\|_{1}\geq \frac{\|K^{F}_{t}(x, \cdot)\|_{2}^{2}}{\|K^{F}_{t}(x, \cdot)
\|_{\infty}}\geq C_{3}>0,
$$
and similarly for $\alpha=\infty$.

Note, that if $q<2<r$, and $ 0<\theta<1$ is  such that $\theta/q+(1-\theta)/r=1/2$, then by the H\"{o}lder inequality
\begin{equation}\label{Hold}
\|K^{F}_{t}(x, \cdot)\|_{2}\leq \|K^{F}_{t}(x, \cdot)
\|_{q}^{\theta}\|K^{F}_{t}(x, \cdot)
\|_{r}^{1-\theta}.
\end{equation}

Assume now that  $2<\alpha\leq \infty$. Then  for $q=1, \>\>r=\alpha$,  we have $\alpha^{\prime}<2(1-\theta)$ and using lower and upper estimates for $p=2$ and $p=1$ respectively we obtain
$$
\|K^{F}_{t}(x, \cdot)\|_{\alpha}\geq \frac{  \|K^{F}_{t}(x, \cdot)   
\|_{2}^{1/(1-\theta)}    }{          \|K^{F}_{t}(x, \cdot)\|_{1}^{\theta/(1-\theta)}       }\geq C_{4}t^{-s/\alpha^{\prime}}
$$
for some $C_{4}>0$.

The case $0\leq\alpha<2$  is handled in a similar way by setting in (\ref{Hold}) $q=\alpha,\>\>r=\infty.$ Lemma is proved.

\end{proof}

The next Theorem is playing an important role in this paper (see also \cite{gm}, \cite{gp1}).

\begin{thm}
\label{wavprop}
Suppose   that for a sufficiently large $J\in \mathbb{N}$ a function $\psi(\xi) = F(\xi^2)$ belongs to $ \mathcal{S}_{J}(\bf{R})$,  is even, and satisfies $\ supp\  \hat{\psi} \subseteq(-1,1)$.  
For $t > 0$, let
$K_t^{F}(x,y)$ be the kernel of $\psi(t\sqrt{L}) = F(t^2L)$.  Then for some $C_0 > 0$, if $d(x,y) > C_0t$,
then $K_t^{F}(x,y) = 0$.
\end{thm}

\begin{proof}

First, let us formulate  the finite speed of propagation property for the wave equation (we closely follow  Theorem 4.5 (iii) in Ch. IV of \cite{Tay}).

Suppose that $L_{{\bf R}^{s}}$ is a second-order differential operator on an open set ${\bf R}^{s}$ in ${\bf R}^{s}$, that $L_1$ is elliptic, and
in fact that, for some $c > 0$,  its
principal symbol $\sigma_2(L_{{\bf R}^{s}})(x,\xi) \geq c^2|\xi|^2$, for all $(x,\xi) \in {\bf R}^{s} \times {\bf R}^{s}$.  
Suppose that $U \subseteq {\bf R}^{s}$ is open, and that $\overline{U} \subseteq {\bf R}^{s}$.  
Then if $\ supp \ h,\ g \subseteq Q \subseteq U$, where $Q$ is compact, then any solution $u$ of
\begin{align}
\left(\frac{\partial^2}{\partial t^2} + L_{{\bf R}^{s}}\right)\phi = 0 \\
\phi(0,x) = h(x)\\
\phi_t(0,x) = g(x)
\end{align}
on $U$ satisfies $\ supp \ \phi(t,\cdot) \subseteq \{x: $ dist $(x, Q) \leq |t|/c\}$.\\
\ \\

It is an easy consequence of this that a similar result holds on manifolds (see explanations in \cite{gp1} and \cite{gp3}). 
Let   $L$ be a smooth elliptic second-order non-negative operator on a manifold ${\bf M}$ and consider  the problem
\begin{equation}\label{syst1}
\left(\frac{\partial^2}{\partial t^2} + L\right)\phi = 0 \\
\end{equation}
\begin{equation}
\phi(0,x) = h(x)\\
\end{equation}
\begin{equation}\label{syst3}
\phi_t(0,x) = 0
\end{equation}
on ${\bf M}$.   It is easy to verify that if  $u_l$ form an orthonormal basis of
eigenfunctions of $L$, with corresponding eigenvalues $\lambda_l$ and
$$
h(x) = \sum_l a_l  u_l(x),\>\>\>a_{l}=\int_{\bf M}h(y)u_{l}(y)dy,
$$
 then
the solution to (\ref{syst1})-(\ref{syst3})  is

\begin{equation}\label{cos}
\phi(x,t)=\sum_{l}\left[a_l \cos \left(t\sqrt{L} \right) u_l\right](x)=\left[ \cos\left(t\sqrt{L}\right)h\right](x),
\end{equation}
or
\[ \phi(x,t) = \sum_{l}a_l \cos \left(t\sqrt{\lambda_l} \right) u_l(x). 
\]
To prove Theorem  it suffices to note that for some $c$
\begin{equation}
\label{wavetrck}
\left[\psi(t{\sqrt L})h\right](x) = c\int_{-1}^{1} \widehat{\psi}(s) \left[\cos\left(s t{\sqrt L}\right)h\right](x)ds
\end{equation}
for any $h \in C^{\infty}({\bf M})$.  This formula follows from the eigenfunction expansion
of $h$ and the Fourier inversion formula. Indeed, since $\widehat{\psi}$ is even and $supp \ \widehat{\psi} \subset (-1, 1)$ we have
$$
\int_{-1}^{1} \widehat{\psi}(s) \left[\cos(s t{\sqrt L})h \right](x)ds=\int_{-1}^{1} \widehat{\psi}(s) \left[\cos(s t{\sqrt L})h + i\ \sin(s t{\sqrt L})h \right](x)ds=
$$
$$
\int_{-\infty}^{\infty}\widehat{\psi}(s)\int_{{\bf M}}\sum_{l}e^{ist\sqrt{\lambda_{l}}}u_{l}(x)u_{l}(y)h(y)dy ds=
$$
$$
\int_{{\bf M}}\sum_{l}\left(\int_{-\infty}^{\infty}\widehat{\psi}(s)e^{ist\sqrt{\lambda_{l}}}ds\right)u_{l}(x)u_{l}(y)h(y)dy=
$$
$$
\int_{{\bf M}}\sum_{l}\psi(t\sqrt{\lambda_{l}})u_{l}(x)u_{l}(y)h(y)dy=\left[\psi(t{\sqrt L})h\right](x).
$$
We also note
\begin{equation}\label{formula}
\left[\psi(t{\sqrt L})h\right](x)=\int_{{\bf M}}\sum_{l}\psi(t\sqrt{\lambda_{l}})u_{l}(x)u_{l}(y)h(y)dy=
$$
$$
\int_{{\bf M}}\sum_{l}    F\left(t^{2}\lambda_{l}\right)    u_{l}(x)u_{l}(y)h(y)dy=\int_{\bf M}K_{t}^{F}(x,y)h(y)dy=F\left(t^{2}L\right)h(x),
\end{equation}
where
$$
\sum_{l}F\left(t^{2}\lambda_{l}\right)u_{l}(x)u_{l}(y)=K_{t}^{F}(x,y).
$$
Let us summarize. Since according to (\ref{cos}) the function $\phi(x,t)= \cos\left(t\sqrt{L}\right)h(x)$ is the solution to (\ref{syst1})-(\ref{syst3}) the finite speed of propagation principle implies that if $h$ has support in a  set $Q\subset {\bf M}$ then 
for every $t>0$ the function $\cos\left(t\sqrt{L}\right)h(x)$ has  support in the set $\{x: $ dist $(x, Q) \leq  C|t|\}$ where $C$ is independent on $Q$.

Consider a function $h\in C^{\infty}({\bf M})$ which is supported in a ball $B_{\epsilon}(y)$ whose center is an $y\in {\bf M}$ and radius is a small $\varepsilon>0$. By (\ref{wavetrck}) and (\ref{formula}) the function $\psi\left(t\sqrt{L}\right)h(x)=F\left(t^{2}L\right)h(x)$ and the kernel $K_{t}^{F}(x,y)$ (as a function in $x$) both have support in the same set 
$$
\{x:  dist  \left(x, B_{\varepsilon}(y)\right) \leq  C|t|\}.
$$
 Since $dist (x, y)=dist \left(x, B_{\varepsilon}(y)\right)+\varepsilon$ we obtain that for any $\varepsilon >0$ the support of $K_{t}^{F}(x,y)$ is in the set 
 $$
 \{x:  dist  \left(x, y\right) \leq  C|t|+\varepsilon\}.
 $$
  Theorem is proven.

\end{proof}

\section{Discretization and reduction to finite-dimensional spaces}

\begin{lem}
\label{balls}
Let ${\bf M}$  be a compact Riemannian manifold. For each positive integer $N$ with $2N^{-1/s} < diam \ {\bf M}$, there exists a collection of disjoint balls
${\mathcal A}^N = \left\{B\left(x_i^N, N^{-1/s}\right)\right\}$, such that the balls with the same centers and 3 times the radii cover ${\bf M}$,
and such that $P_N := \#{\mathcal A}^N \asymp N$.  
\end{lem}

\begin{proof}
  We need only let ${\mathcal A}^N$ be a maximal disjoint collection of balls of radius $N^{-1/s}$.  Then surely
the balls with the same centers and $3$ times the radii cover ${\bf M}$.  Thus by disjointness
\[ \mu({\bf M}) \geq \sum_{i=1}^{P_N}  \mu\left(B\left(x_i^N, N^{-1/s}\right)\right) \gg \sum_{i=1}^{P_N} 1/N = P_N/N, \]
while by the covering property
\[ P_N/(3^sN) \gg \sum_{i=1}^{P_N} \mu\left(B\left(x_i^N, 3N^{-1/s}\right)\right) \geq \mu({\bf M}) \]
so that $P_N \asymp N$ as claimed.

\end{proof}

Now we formulate and sketch the proof of the following Lemma \ref{disfns}.  See \cite{gp3} for more details.

In what follows we consider collections of balls ${\mathcal A}^N$ as in  Lemma \ref{balls}.

\begin{lem}
\label{disfns}
Let ${\bf M}$  be a compact Riemannian manifold. Then there are smooth functions $\varphi_i^N$ ($2N^{-1/s} < diam \ {\bf M}$,
$1 \leq i \leq P_N$), as follows:
\begin{enumerate}

\item  supp $\varphi_i^N \subseteq B_i^N := B(x_i^N, N^{-1/s})$;

\item for $1 \leq p \leq \infty$, $\|\varphi_i^N\|_p \asymp N^{-1/p}$, with constants independent of $i$ or $N$.

\end{enumerate}

\end{lem}

\begin{proof}  For a sufficiently large $J\in \mathbb{N}$  let $h_0(\xi) = F_0(\xi^2)$ be an even element of ${\mathcal S}_{J}(\mathrm{R})$ with supp\ $\hat{h}_0 \subseteq (-1,1)$. 
For a postitive integer $Q$ yet to be chosen, let $F(\lambda) =  \lambda^Q F_0(\lambda)$, and set \begin{equation}\label{Q}
h(\xi) = F(\xi^2) = \xi^{2Q}F_0(\xi^2),
\end{equation}
so that $\hat{h} = c\partial^{2Q}\hat{h_0}$ still has support contained in $(-1,1)$.
Thus, by Theorem \ref{wavprop},
there is a $C_0 > 0$ such that
for $t > 0$, the kernel $K_t^{F}(x,y)$ of $h(t\sqrt{\mathcal L}) = F(t^2{\mathcal L})$ has the property that
$K_t^{F}(x,y) = 0$ whenever $d(x,y) > C_0t$.   Thus if $t = N^{-1/s}/2C_0$, 
\begin{equation}\label{phi}
\varphi_i^N(x) := \frac{1}{N}K_t^{F}(x_i^N,x) 
\end{equation}
satisfies (1).  By Theorem \ref{lowerestimate}, $\|\varphi_i^N\|_p \asymp N^{-1}(N^{-1/s})^{-s/p'} = N^{-1/p}$, so
(2) holds. Lemma is proven.
\end{proof}

Let $\varphi_{i}^{N}$ be the same as above. We  consider their span 
\begin{equation}\label{M}
\mathcal{H}^{N}_{p}=\left\{\sum_{i=1}^{P_N}a_i\varphi_i^N  :  a=(a_1,... ,a_{P_N}) \in \mathbf{R}^{P_N}\right\}
\end{equation}
as a finite-dimensional Banach space $\mathcal{H}^{N}_{p}$ with the norm 
\begin{equation}\label{H}
\left\|\sum_{i=1}^{P_N}a_i\varphi_i^N\right\|_{\mathcal{H}^{N}_{p}}=\left\|\sum_{i=1}^{P_N}a_i\varphi_i^N\right\|_{L_{p}({\bf M})}\asymp
CN^{-1/p}\|a\|_{p},
\end{equation}
where $C$ is independent on $N$. Clearly, for any $r>0$ the operator $L^{r/2}$ maps $\mathcal{H}^{N}_{p}$ onto the span 
$$
\mathcal{M}^{N}_{p}=\left\{\sum_{i=1}^{P_N}a_iL^{r/2}\varphi_i^N  :  a=(a_1,... ,a_{P_N}) \in \mathbf{R}^{P_N}\right\},
$$
which we will consider with the norm
$$
\left\|\sum_{i=1}^{P_N}a_iL^{r/2}\varphi_i^N\right\|_{L_{p}({\bf M})},
$$
and will denote as $\mathcal{M}^{N}_{p}\subset L_{p}({\bf M})$.
Our next goal is to estimate norm of $L^{r/2}$ as an operator from the Banach space $\mathcal{H}^{N}_{p}$ onto Banach space $\mathcal{M}^{N}_{p}$.
\begin{lem}\label{3}

If $\varphi_{i}^{N}$ are the same as in Lemma \ref{disfns}  then
 for $1 \leq p \leq \infty$, and $r > 0$, 
\begin{equation}\label{inequality}
\left\|\sum_{i=1}^{P_N}a_iL^{r/2}\varphi_i^N\right\|_{L_{p}({\bf M})} \leq 
CN^{\frac{r}{s}-\frac{1}{p}}\|a\|_p,
\end{equation}

with $C$ independent of $a = (a_1,... ,a_{P_N}) \in \mathbf{R}^{P_N}$, $p$ or $N$.

\end{lem}

\begin{proof}
 By the Riesz-Thorin interpolation
theorem, we need only to verify the estimates  for $p = 1$ and $p=\infty$.  For  $t = N^{-1/s}/2C_0$ and $\varphi_{i}^{N}$ defined in (\ref{phi}) we have
\begin{equation}
\label{kgt}
L^{r/2}\varphi_i^N = N^{-1}t^{-r}\sum_l (t^2\lambda_l)^{r/2} F(t^2\lambda_l) u_l(x_i^N)u_l(x)
= CN^{\frac{r}{s}-1}K^G_t(x_i^N,x),
\end{equation}
where $G(\lambda) = \lambda^{r/2}F(\lambda)$ and $F$ is defined in (\ref{Q}).  Clearly, for a fixed $r>0$ function $G$ belongs to a certain $\mathcal{S}_{J_{0}}({\bf R}^{+})$ for some $J_{0}\in \mathbb{N}$ if $Q$ in (\ref{Q}) is sufficiently large. Note that $C$ in (\ref{kgt}) is independent of $N, i$ or $t$.  Thus, by (\ref{kgt}) and Lemma \ref{Lalphok},
for $p=1$ we have $\|L^{r/2}\varphi_i^N\|_1 \leq CN^{\frac{r}{s}-1}$, with $C$ independent of $i,N$.  
It proves  (\ref{inequality}) for $p=1$.  As for $p = \infty$, we again set $t = N^{-1/s}/2C_0$.
By Lemma \ref{kersize} and (\ref{phi}) , we have that for any $x$, 
\begin{equation}
\left|\sum_{i=1}^{P_N}a_i L^{r/2}\varphi_i^N(x)\right|  \leq  
CN^{\frac{r}{s}-1} \|a\|_{\infty}\sum_{i=1}^{P_N}\frac{t^{-s}}{(1+d(x_i^N,x)/t)^{s+1}}.
\end{equation}
Since $t^{-s}=\left(2C_{0}\right)^{s}N\asymp \mu\left(B_{i}^{N}\right)N^{2}$, we obtain
\begin{equation}
 \left|\sum_{i=1}^{P_N}a_i L^{r/2}\varphi_i^N(x)\right|\leq CN^{\frac{r}{s}+1} \|a\|_{\infty}\sum_{i=1}^{P_N}\frac{\mu(B_i^N)}{(1+d(x_i^N,x)/t)^{s+1}} .
 \end{equation}
 The triangle inequality shows that for 
all $x \in {\bf M}$, all $t > 0$, all $i$ and $N$, and all $y \in B_i^N$, one has $(1+d(y,x)/t) \leq C(1+d(x_i^N,x)/t)$ with $C$ independent of $x,y,t,i,N$. Combining this with  (\ref{intest}) we finally obtain
 \begin{equation}
  \left|\sum_{i=1}^{P_N}a_i L^{r/2}\varphi_i^N(x)\right|\leq CN^{\frac{r}{s}+1} \|a\|_{\infty}\int_{\bf M}\frac{dy}{(1+d(y,x)/t)^{s+1}}\leq  CN^{\frac{r}{s}} \|a\|_{\infty}.
\end{equation}
Lemma \ref{3} is proved.

\end{proof}

The next step is to reduce our main problem to a finite-dimensional situation.

Let us remained that we are using the following  notations.  $S_{n}$ will stay for either Kolmogorov $n$-width $d_{n}$ or linear $n$-width $\delta_{n}$;
 the notation $s_{n}$ will be used for either $d_{n}$ or Gelfand $n$-width $d^{n};
\>\> S^{n}$ will be used for either $d_{n}, \>\> d^{n},$ or $\delta_{n}$.

Below we will need  the following relations (see \cite{LGM}, pp. 400-403,):
\begin{equation}
\label{pupqdn}
S^n(H_{1}, Y) \leq S^n(H, Y),
\end{equation}
if $H_{1}\subset H$, and 
\begin{equation}
\label{pupqdn2}
d^{n}(H, Y)=d^{n}(H, Y_{1}), \>\>\>
S_n(H, Y) \leq S_n(H, Y_{1}), \>\>H\subset Y_{1}\subset Y, 
\end{equation}
where  $Y_{1}$ is  a subspace of $ Y$.
  Moreover, the following inequality holds
\begin{equation}
\label{linmax}
\delta_n(H,Y) \geq \max(d_n(H,Y),d^n(H,Y)).
\end{equation}

In what follows we are using notations of Lemmas \ref{balls}-\ref{3}.
\begin{lem}
\label{lowbd1}
For $1 \leq p,q \leq \infty$, if  $s_n = d_n$ or $d^n$, then
\begin{equation}
\label{lowbd1way}
s_n\left(B^r_p({\bf M}), L_q({\bf M})\right) \geq CN^{-\frac{r}{s}+\frac{1}{p}-\frac{1}{q}}s_n(b_p^{P_N},\ell_q^{P_N}),
\end{equation}
for any sufficiently large $n,N$, with $C$ independent of $n,N$.
\end{lem}

\begin{proof}
 With the $\varphi_i^N$ as in Lemma \ref{disfns}, we consider the space of functions of the form
\begin{equation}
\label{hndf}
g_a = \sum_{i=1}^{P_N} a_i \varphi_i^N,
\end{equation}
for $a = (a_1,\ldots,a_{P_N}) \in {\bf R}^{P_N}$.  By Lemma \ref     {disfns}  
and the disjointness of the $B_i^N$,
\begin{equation}
\label{gaq}
\|g_a\|_{q} \asymp N^{-1/q}\|a\|_q,
\end{equation}
with constants independent of $N$ or $a$.  
By Lemma \ref{3}  for some $c > 0$, if we set $\epsilon = \epsilon_N = cN^{-\frac{r}{s}+\frac{1}{p}},$
and if $a \in \epsilon b_p^{P_N}$, then $g_a \in B^r_p$.  Thus,
\begin{equation}
\label{gndf}
\mathcal{G}^{N}_{p} := \{g_a \in \mathcal{H}^{N}_{p}: a \in \epsilon b_p^{P_N}\} \subseteq B^r_p.
\end{equation}
For the Gelfand widths, it is a consequence of the Hahn-Banach theorem,
that if $K \subseteq X \subseteq Y$, where $X$ is a subspace of the normed space $Y$, then
$d^n(K,X) = d^n(K,Y)$ for all $n$.   Thus, using (\ref{H}) we obtain
$$
d^{n}\left(B^r_p({\bf M}),L_q({\bf M})\right)  \geq d^n(\mathcal{G}^{N}_{p},L_q) = d^n(\mathcal{G}^{N}_{p},\mathcal{H}^{N}_{q}) \geq 
$$
$$
CN^{-1/q}d^n(\epsilon_N b_p^{P_N},\ell_q^{P_N})=CN^{-r/s+1/p-1/q}d^n(b_p^{P_N},\ell_q^{P_N})
$$
for some $C$ independent of $n, N$.  This proves the lemma for the 
Gelfand widths.

For the Kolmogorov widths, for the same reason, we need only show that
\begin{equation}
\label{kolgd}
d_n(B^r_p,L_q) \geq Cd_n(\mathcal{G}^{N}_{p},\mathcal{H}^{N}_{q}).
\end{equation}
with $C$ independent of $n,N$.

To this end we define the projection operator $\Pi_N: L_q \to \mathcal{H}^{N}_{q}$ by
\[\Pi_Nh = g_a,\ \ \ \ \ \ \ \ \ \ \ \ \ \mbox{ where } a_i = \frac{\int h\varphi_i^N}{\|\varphi\|^2_2}.  \]
By Lemma \ref{disfns} and H\"older's inequality, we have here that each $|a_i| \leq C\|h\chi_i^N\|_q N^{1-1/q'}$,
where $\chi_i^N$ is the characteristic function of $B_i^N$.  By (\ref{gaq}) and the disjointness of the $B_i^N$, we have that
\begin{equation}
\label{qngd}
\|\Pi_N h\|_q = \|g_a\|_q \asymp N^{-1/q}\|a\|_q \leq cN^{1-1/q-1/q'}\|h\|_q = c\|h\|_q,
\end{equation}
with $C$ independent of $n,N$.  

Accordingly, for any $g \in \mathcal{H}^{N}_{q}$ and $h \in L_q$, we have that 
$$
\|g-\Pi_Nh\|_q = \|\Pi_Ng-\Pi_Nh\|_q \leq c\|g-h\|_q.
$$
  Thus, if
$K$ is any subset of $\mathcal{H}^{N}_{q}$, $d_n(K,L_q) \geq c^{-1}d_n(K,\mathcal{H}^{N}_{q})$.  In particular
\[ d_{n}\left(B^r_p({\bf M}),L_q({\bf M})\right)  \geq d_n(\mathcal{G}^{N}_{p},L_q) \geq c^{-1}d_n(\mathcal{G}^{N}_{p},\mathcal{H}^{N}_{q}). \]
This establishes (\ref{kolgd}), and completes the proof.

\end{proof}

\section{Proof of the main result}\label{Proof}

In this section we will prove Theorem \ref{Main}.

  We will need several facts about widths. First, say $p \geq p_1$, $q \leq q_1$, and $S^n = d_n, d^n$ or 
$\delta_n$.  One then has the following two evident facts
\begin{equation}
\label{pupqdn}
S^n\left(B^r_p({\bf M}),L_q({\bf M})\right)  \leq CS^n\left(B^r_{p_{1}}({\bf M}),L_{q_{1}}({\bf M})\right) 
\end{equation}
with $C$ independent of $n$, while
\begin{equation}
\label{pdnqup}
S^n(b_p^Q,\ell_q^Q) \geq CS^n(b_{p_1}^Q,\ell_{q_1}^Q)
\end{equation}
with $C$ independent of $n, Q$.  

By Lemma \ref{balls}, we may choose $\nu > 0$ such that $P_{\nu n} \geq 2n$ for all sufficiently
large $n$.  In this proof we will always take $N = \nu n$.  We consider the various ranges of $p,q$
separately:

\begin{enumerate}
\item
$1\leq q \leq p\leq \infty$.\\
\ \\
In this case, we note that if $S^n = d_n, d^n$ or $\delta_n$, then by (\ref{pupqdn}),
\begin{equation}
\label{pqinf1}
S^n\left(B^r_p({\bf M}),L_q({\bf M})\right)  \geq CS^n\left(B^r_\infty({\bf M}),L_1({\bf M})\right) .
\end{equation}
On the other hand, if $s_n = d_n$ or $d^n$, then by (3.1) on page 410 of \cite{LGM},
$s_n(b_{\infty}^{P_N},\ell_{1}^{P_N}) = P_N - n \geq n$.  By this, (\ref{pqinf1}) and
Lemma \ref{lowbd1}, we find that
\[ s_n\left(B^r_p({\bf M}),L_q({\bf M})\right)  \gg n^{-\frac{r}{s}-1}n = n^{-\frac{r}{s}} \]
first for $s_n = d_n$ or $d^n$ and then for $\delta_n$, by (\ref{linmax}).  This
completes the proof in this case.
\item
$1 \leq p \leq q \leq 2$.\\
\ \\ 
In this case, for the Gelfand widths we just observe, by (\ref{pupqdn}), that
\begin{equation}
\label{pqp1}
d^{n}\left(B^r_p({\bf M}),L_q({\bf M})\right)  \geq Cd^{n}\left(B^r_p({\bf M}),L_p({\bf M})\right) \gg n^{-\frac{r}{s}}
\end{equation}
by case 1.  For the Kolmogorov widths we observe, by Lemma \ref{lowbd1} and (\ref{pdnqup}), that
\begin{equation}
\label{pq12}
d_{n}\left(B^r_p({\bf M}),L_q({\bf M})\right)  \gg n^{-\frac{r}{s}+\frac{1}{p}-\frac{1}{q}}d_n(b_{p}^{P_N},\ell_q^{P_N}) \gg
$$
$$
n^{-\frac{r}{s}+\frac{1}{p}-\frac{1}{q}}d_n(b_{1}^{P_N},\ell_2^{P_N}) \gg
n^{-\frac{r}{s}+\frac{1}{p}-\frac{1}{q}},
\end{equation}
since, by (3.3) of page 411 of \cite{LGM}, $d_n(b_{1}^{P_N},\ell_2^{P_N}) = \sqrt{1-n/P_N} \geq 1/\sqrt{2}$.
Finally, for the linear widths, we have by (\ref{linmax}), that
\[  \delta_{n}\left(B_{p}^{r}({\bf M}), L_{q}({\bf M})\right) \gg n^{-\frac{r}{s}+\frac{1}{p}-\frac{1}{q}}. \]
This completes the proof in this case.
\item
$2 \leq p \leq q\leq \infty$.\\
\ \\ 
In this case, for the Kolmogorov widths we just observe, by (\ref{pupqdn}), that
\begin{equation}
\label{pqp1}
d_{n}\left(B^r_p({\bf M}),L_q({\bf M})\right)  \geq Cd_{n}\left(B^r_p({\bf M}),L_p({\bf M})\right)  \gg n^{-\frac{r}{s}}
\end{equation}
by case 1.  For the Gelfand widths we observe, by Lemma \ref{lowbd1} and (\ref{pdnqup}), that
\begin{equation}
\label{pq12}
d^{n}\left(B^r_p({\bf M}),L_q({\bf M})\right)  \gg n^{-\frac{r}{s}+\frac{1}{p}-\frac{1}{q}}d^n(b_{p}^{P_N},\ell_q^{P_N}) \gg
$$
$$
n^{-\frac{r}{s}+\frac{1}{p}-\frac{1}{q}}d^n(b_{2}^{P_N},\ell_{\infty}^{P_N}) \gg
n^{-\frac{r}{s}+\frac{1}{p}-\frac{1}{q}},
\end{equation}
since, by (3.5) on page 412 of \cite{LGM}, 
$$
d^n(b_{2}^{P_N},\ell_{\infty}^{P_N}) = \sqrt{1-n/P_N} \geq 1/\sqrt{2}.
$$
Finally, for the linear widths, we have by (\ref{linmax}), that
\[  \delta_{n}\left(B_{p}^{r}({\bf M}), L_{q}({\bf M})\right) \gg n^{-\frac{r}{s}+\frac{1}{p}-\frac{1}{q}}. \]
This completes the proof in this case.
\item
$1 \leq p \leq 2 \leq q \leq \infty$.\\
\ \\
If  $1 \leq \alpha \leq \alpha_1 \leq \infty$, then by H\"older's inequality, for every $a=(a_{1},..., a_{P_{N}})$
\begin{equation}
\label{hold1}
\|a\|_{\alpha} \leq P_{N}^{\frac{1}{\alpha}-\frac{1}{\alpha_1}}\|a\|_{\alpha_1}.
\end{equation}
This implies that
\begin{equation}
\label{hold2}
b_{\alpha_1}^{P_{N}} \subseteq P_{N}^{\frac{1}{\alpha_1}-\frac{1}{\alpha}}b_{\alpha}^{P_{N}}.
\end{equation}
From Lemma \ref{lowbd1}, (\ref{pdnqup}) and (\ref{hold1}), we find that
\begin{equation}
\label{1p2qdn1}
d_{n}\left(B^r_p({\bf M}),L_q({\bf M})\right)  \gg n^{-\frac{r}{s}+\frac{1}{p}-\frac{1}{q}}d_n(b_p^{P_N},\ell_q^{P_N})
\gg 
$$
$$
n^{-\frac{r}{s}+\frac{1}{p}-\frac{1}{q}}d_n(b_1^{P_N},\ell_q^{P_N})
\gg n^{-\frac{r}{s}+\frac{1}{p}-\frac{1}{2}}d_n(b_1^{P_N},\ell_2^{P_N})
\gg n^{-\frac{r}{s}+\frac{1}{p}-\frac{1}{2}}.
\end{equation}
From Lemma \ref{lowbd1}, (\ref{pdnqup}) and (\ref{hold2}), we find that
\begin{equation}
\label{1p2qdn2}
d^{n}\left(B^r_p({\bf M}),L_q({\bf M})\right)  \gg n^{-\frac{r}{s}+\frac{1}{p}-\frac{1}{q}}d^n(b_p^{P_N},\ell_q^{P_N})
\gg 
$$
$$
n^{-\frac{r}{s}+\frac{1}{p}-\frac{1}{q}}d^n(b_p^{P_N},\ell_{\infty}^{P_N})
\gg n^{-\frac{r}{s}+\frac{1}{2}-\frac{1}{q}}d^n(b_2^{P_N},\ell_{\infty}^{P_N})
\gg n^{-\frac{r}{s}+\frac{1}{2}-\frac{1}{q}}.
\end{equation}
Finally, from (\ref{1p2qdn1}), (\ref{1p2qdn2}) and (\ref{linmax}),
\begin{equation}
\label{1p2qdn3}
\delta_{n}\left(B_{p}^{r}({\bf M}), L_{q}({\bf M})\right) \gg \max\left(n^{-\frac{r}{s}+\frac{1}{p}-\frac{1}{2}},n^{-\frac{r}{s}+\frac{1}{2}-\frac{1}{q}}\right).
\end{equation}
This completes the proof of our main Theorem \ref{Main}.

\end{enumerate}

\end{document}